\documentclass[12pt]{article}

\usepackage[T1]{fontenc}
\usepackage[active]{srcltx}
\usepackage{lmodern}
\usepackage{amsmath,amstext,amssymb,graphicx,graphics,color}
\usepackage{mathtools,mathrsfs}
\usepackage{tikz}
\usetikzlibrary{decorations.markings,calc,fit,intersections,arrows,matrix,trees,positioning,shapes,tikzmark,backgrounds}
\usepackage{mleftright}
\usepackage{pgfplots}
\usepackage{theorem}
\usepackage{cite}               
\usepackage{hyperref}           
\usepackage{bbm}                
\usepackage{fancybox}           
\usepackage[section]{placeins}  
\usepackage[font=footnotesize]{caption}
\usepackage{subcaption}

\usepackage{arydshln} 
\tikzset{
    invisible/.style={opacity=0},
    visible on/.style={alt={#1{}{invisible}}},
    alt/.code args={<#1>#2#3}{%
      \alt<#1>{\pgfkeysalso{#2}}{\pgfkeysalso{#3}}%
  }
}
\theorembodyfont{\normalfont\slshape} 

\newtheorem{theorem}{Theorem}
\newtheorem{proposition}{Proposition}[section]

\theoremstyle{break} 

\newenvironment{proof}%
{{\par\noindent \bf Proof. \nobreak}}%
{\nobreak \removelastskip \nobreak \hfill $\Box$ \medbreak}
{{\par\noindent \bf Proof \nobreak}}%
{\nobreak \removelastskip \nobreak \hfill $\Box$ \medbreak}
{{\par\noindent \bf Proof lemma. \nobreak}}%
{\nobreak \removelastskip \nobreak \bf End proof lemma. \medbreak}
\newenvironment{remark}{\par \medskip \noindent {\bf Remark. }\nobreak}{\par \medskip}
\usepackage{fancyhdr}
 \fancypagestyle{plain}{
  \fancyhead{}
  
}

\def\paragraph#1{{\bf #1\ }}
\newcommand{\expo}{\mathrm{e}}

\newcommand{\dd}{\mathrm{d}}

\newcommand{\D}{\mathrm{D}}
\newcommand{\E}{\mathrm{E}}

\usepackage[utf8]{inputenc}
\usepackage{unicode_character}

\title{Uniform propagation of chaos for a dollar exchange econophysics model}

\author{Roberto Cortez \footnotemark[1] \and Fei Cao\footnotemark[2]}

\begin{document}
\maketitle

\footnotetext[1]{Universidad Andrés Bello - Departamento de Matemáticas, Santiago, Chile}
\footnotetext[2]{University of Massachusetts Amherst - Department of Mathematics and Statistics, Amherst, MA 01003, USA}

\tableofcontents


\begin{abstract}


 We study the poor-biased model for money exchange introduced in \cite{cao_derivation_2021}: agents are being randomly picked at a rate proportional to their current wealth, and then the selected agent gives a dollar to another agent picked uniformly at random. Simulations of a stochastic system of finitely many agents as well as a rigorous analysis carried out in \cite{lanchier_rigorous_2017,cao_derivation_2021} suggest that, when both the number of agents and time become large enough, the distribution of money among the agents converges to a Poisson distribution. In this manuscript, we establish a uniform-in-time propagation of chaos result as the number of agents goes to infinity, which justifies the validity of the mean-field deterministic infinite system of ordinary differential equations as an approximation of the underlying stochastic agent-based dynamics.
\end{abstract}


\section{Introduction}\label{sec:sec1}
\setcounter{equation}{0}

In this manuscript, we study a simple mechanism for money exchange in a closed economical system, meaning that there are a fixed number of agents, denoted by $N$, with an (fixed) average number of dollars $\mu \in \mathbb{N}_+$. We denote by $S_i(t) \in \mathbb{N}$ the amount of dollars the agent $i$ has at time $t$. Since it is a closed economical system, we have
\begin{equation}\label{eq:preserved_sum}
S_1(t)+ \cdots +S_N(t) = N\mu = \textrm{Constant} \qquad \text{for all } t\geq 0.
\end{equation}
Specifically, we consider the so-called poor-biased dollar exchange model investigated in \cite{cao_derivation_2021}: at random times (generated by an exponential law), an agent $i$ is picked at a rate which is proportional to its current wealth, and he or she will give one dollar to another agent $j$ picked uniformly at random. In particular, if agent $i$ does not have at least one dollar, then he/she will never be picked to give. Mathematically, the update rule of this $N$-agent system can be represented by
\begin{equation}
\label{poor_biased_exchange}
\textbf{poor-biased exchange:} \qquad (S_i,S_j)~ \begin{tikzpicture} \draw [->,decorate,decoration={snake,amplitude=.4mm,segment length=2mm,post length=1mm}]
(0,0) -- (.6,0); \node[above,red] at (0.3,0) {\small{$λ\,S_i/N$}};\end{tikzpicture}~  (S_i-1,S_j+1).
\end{equation}

We emphasize that in order to ensure that the rate of a typical agent giving a dollar per unit time is of order $1$ (so that the correct mean-field analysis as $N→+∞$ can be carried out), the rate appearing in \eqref{poor_biased_exchange} is set to be $λ\,S_i/N$ instead of $λ\,S_i$. 

\begin{figure}[!htb]
  \centering
  \includegraphics[scale = 0.7]{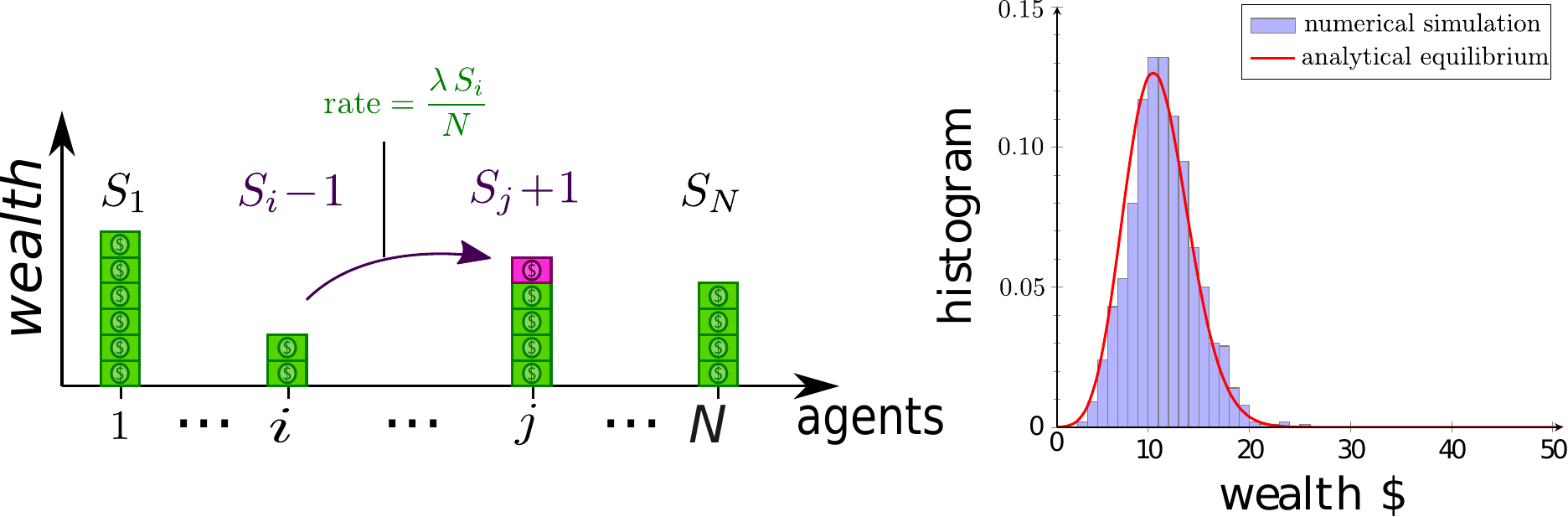}
  \caption{{\bf Left:} Illustration of the poor-biased dollar exchange model: at random time, one dollar is passed from a ``giver'' $i$ to a ``receiver'' $j$ at a rate proportional to the amount of dollars the ``giver'' $i$ has. {\bf Right:} The distribution of wealth for the poor-biased dynamics after $2,000$ unit of time with the average amount of dollar per agent $\mu = 10$, this distribution is well-approximated by a Poisson distribution with mean value $\mu = 10$.}
  \label{fig:illustration_model}
\end{figure}

We illustrate the dynamics in figure \ref{fig:illustration_model}-left. The main task is to identify the limiting distribution of money when both the number of agents $N$ and time $t$ become large enough. We illustrate numerically in figure \ref{fig:illustration_model}-right the simulation result using $N=1,000$ agents. Notice that the wealth distribution is well-approximated by a Poisson distribution with mean value $\mu = 10$.

If we denote by ${\bf p}(t)=\left(p_0(t),p_1(t),\ldots\right)$ the law of the process ${S}_1(t)$ as $N \to \infty$, i.e. $p_n(t) = \lim_{N \to \infty} \mathbb P\left(S_1(t) = n\right)$, it has been shown recently in \cite{cao_derivation_2021} that the time evolution of ${\bf p}(t)$ is given by:
\begin{equation}
  \label{eq:law_limit}
  \frac{\dd}{\dd t} {\bf p}(t) = λ \,\mathcal{L}[{\bf p}(t)]
\end{equation}
with:
\begin{equation}
  \label{eq:L}
  \mathcal{L}[{\bf p}]_n:= \left\{
    \begin{array}{ll}
      p_1-\mu\,p_0 & \quad \text{for}~ n=0, \\
      (n+1)\,p_{n+1}+\mu\,p_{n-1}- \left(n+\mu\right)\,p_n & \quad \text{for}~ n \geq 1.
    \end{array}
  \right.
\end{equation}
 A harmless normalization allows us to put $\lambda = 1$ without any loss of generality, which will be implicitly assumed throughout the manuscript. The linear ODE system \eqref{eq:law_limit} is of Fokker-Planck type and hence admits an interpretation in terms of a `gain' and `loss' process, shown in figure \ref{fig:ODE_illustration}.

\begin{figure}[!htb]
\centering
\includegraphics[scale=1.0]{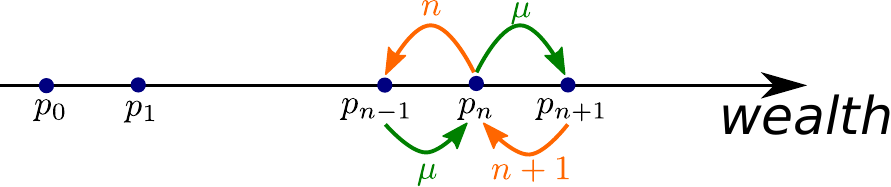}
\caption{Schematic illustration of the limiting ODE system \eqref{eq:law_limit}.}
\label{fig:ODE_illustration}
\end{figure}

The transition from the stochastic $N$-agent dynamics \eqref{poor_biased_exchange} to the infinite system of ODEs \eqref{eq:law_limit} as $N \to \infty$ is accomplished thanks to the notion of \emph{propagation of chaos} \cite{sznitman_topics_1991} and has been rigorously justified in \cite{cao_derivation_2021}. Unfortunately, only a finite time propagation of chaos result is obtained in \cite{cao_derivation_2021}, meaning that the evolution equation \eqref{eq:law_limit} is only guaranteed to be a good approximation of the $N$-agent system over the time span $t \in [0,T]$ with $T > 0$ being arbitrary but fixed. In this work, we aim to establish a uniform in time propagation of chaos for this particular dynamics, which refines the previous short time result.

We first summarize the main result of \cite{cao_derivation_2021} regarding the large time behavior of the linear ODE system in the following proposition and refer interested readers to \cite{cao_derivation_2021} for detailed proofs and discussion.

\begin{proposition}\label{summarize_ode} 
Let ${\bf p}(t)=\{p_n(t)\}_{n\geq 0}$ be the unique solution of \eqref{eq:law_limit} with $p(0) \in V_\mu$, where \[V_\mu:=\biggl\{\,{\bf p} \mathrel{\Big|} p_n \geq 0,~\sum_{n=0}^\infty p_n =1,~\sum_{n=0}^\infty n\,p_n =\mu \,\biggr\}\]
 is the space of probability mass functions on $\mathbb{N}$ with the pre-fixed mean value $\mu$. Then
\begin{equation}
\label{eq:conservation_mass_mean_value}
\sum_{n=0}^\infty \mathcal{L}[{\bf p}]_n =0,~~\textrm{and}~~ \sum_{n=0}^\infty n\,\mathcal{L}[{\bf p}]_n =0.
\end{equation}
In particular, we have ${\bf p}(t) \in V_\mu$ for all $t\geq 0$.  Moreover, the unique equilibrium distribution ${\bf p}^*=\{p^*_n\}_n$ in $V_\mu$ associated with \eqref{eq:law_limit} is given by the following Poisson distribution:
\begin{equation}
\label{eqn:equil_limit}
p^*_n = \frac{\mu^n\,\expo^{-\mu}}{n!},\quad n\geq 0.
\end{equation}
Moreover, if we introduce the following energy functional for each ${\bf p} \in V_\mu$:
\begin{equation}\label{eq:E}
\E[{\bf p}] = \sum\limits_{n=0}^\infty \frac{p^2_n}{p^*_n}
\end{equation}
and \begin{equation}\label{eq:D}
\D[{\bf p}] = \sum\limits_{n=0}^\infty p^*_n\,\left(\frac{p_{n+1}}{p^*_{n+1}} - \frac{p_n}{p^*_n}\right)^2,
\end{equation}
then \begin{subequations}\label{eq:Bakry_Emery}
\begin{align}
\frac{\dd \E[{\bf p}(t) - {\bf p}^*]}{\dd t} &= -2\,\mu\,\D[{\bf p}(t)] \label{eq:subeq_a} \\
\frac{\dd^2 \E[{\bf p}(t) - {\bf p}^*]}{\dd t^2} &\geq -2\,\frac{\dd \E[{\bf p}(t) - {\bf p}^*]}{\dd t} \label{eq:subeq_b}.
\end{align}
\end{subequations}
Consequently, ${\bf p}(t)$ decays exponentially fast towards ${\bf p}^*$ in the sense that
\begin{equation}\label{eq:exponential_decay}
\E[{\bf p}(t) - {\bf p}^*] \leq \E[{\bf p}(0) - {\bf p}^*]\,\expo^{-t}.
\end{equation}
\end{proposition}

\begin{remark}
The exponential decay result \eqref{eq:exponential_decay} is obtained via the celebrated Bakry-Émery approach \cite{bakry_diffusions_1985} and the rate appearing on the right hand side of \eqref{eq:exponential_decay} seems to be half of the sharp rate based on numerical simulations carried out in \cite{cao_derivation_2021}. In other words, numerical experiments suggest that we can strengthen \eqref{eq:exponential_decay} to
\begin{equation}\label{eq:exponential_decay_refined}
\E[{\bf p}(t) - {\bf p}^*] \leq \E[{\bf p}(0) - {\bf p}^*]\,\expo^{-2\,t}.
\end{equation}
\end{remark}

Lastly, we emphasize that the convergence to the Poisson distribution can be also studied from a different point of view. For instance, Lanchier \cite{lanchier_rigorous_2017} investigated a discrete-time analog of the model by sending $t \to \infty$ first prior to sending $N \to \infty$. To give a brief account of this approach, we denote by ${\bf S}(t) = \left(S_1(t),\ldots,S_N(t)\right)$ and recall that the vector ${\bf S}(t)$ is a Markov pure-jump process on the following configuration space
\begin{equation}\label{eq:state_space}
\mathscr{A}_{N,\mu} \coloneqq \big\{{\bf S} \in \mathbb{N}^N \mid \sum_{n=1}^N S_i = N\mu\big\}.
\end{equation}
The key insight behind this approach lies in the fact for any fixed $N \in \mathbb{N}_+$, the dynamics starting from any initial configuration converges (as $t\to \infty$) to the multinomial distribution on $\mathscr{A}_{N,\mu}$ corresponding to taking $N \mu$ independent samples uniformly at random from the set $\{1,\ldots,N\}$, with replacement. More specifically, this distribution, which we denote $\mathcal{M}_N$, is given by
\begin{equation}
\label{eq:multinomial}
\mathcal{M}_N \left({\bf S}\right) \coloneqq \binom{N\mu}{S_1,S_2,\ldots,S_N}\prod\limits_{i \in [N]} \frac{1}{N^{S_i}}.
\end{equation}
This roughly means that each dollar will be equally likely to be in any agent's pocket when time becomes sufficiently large. Then the large population limit $N \to \infty$ can be performed with the help of some basic algebra and combinatorial counting techniques.

We encapsulate the various approaches introduced so far in figure \ref{fig:diagram}.

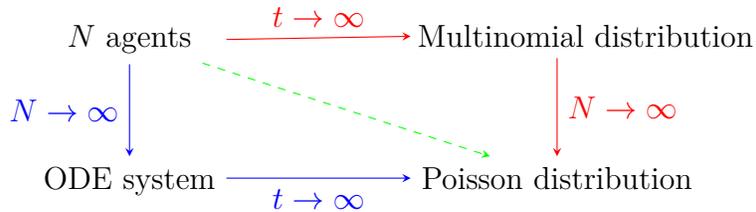
\begin{figure}[!htb]
\centering
\begin{tikzpicture}[scale=0.8]
\tikzset{->-/.style={decoration={
 markings,
 mark=at position #1 with {\arrow{>}}},postaction={decorate}}}
\matrix (m) [matrix of math nodes,row sep=3em,column sep=4em,minimum width=2em]
  {
     N~\text{agents} & \hspace{0.3in}\text{Multinomial distribution} \\
  \text{ODE system} &  \text{Poisson distribution} \\};
  \path[-stealth]
    (m-1-1) edge [blue] node [left] {{\color{blue}{$N \rightarrow \infty$}}} (m-2-1)
    (m-1-1 -| m-2-1.east)  edge [red] node [above] {{\color{red}{$t \rightarrow \infty$}}} (m-1-2 -| m-2-2.west)
    (m-1-1) edge [dashed,green] (m-2-2)
    (m-2-1.east|-m-2-2) edge[blue] node [below] {{\color{blue}{$t \rightarrow \infty$}}} (m-2-2)
    (m-1-2) edge [red] node [right] {{\color{red}{$N \rightarrow \infty$}}} (m-2-2);
            edge [-] (m-2-1);
\end{tikzpicture}
\caption{Roadmap for proving convergence results. The approach of taking the large time limit $t\to \infty$ before taking the large population limit $N \to \infty$ is adapted in \cite{lanchier_rigorous_2017}. An alternative approach is to send $N \to \infty$ first before investigating the large time asymptotic \cite{cao_derivation_2021}.}
\label{fig:diagram}
\end{figure}

Although we will only investigate a specific binary exchange models in the present work, the literature on other types of econophysics models based on different exchange rules is vast. To name a few, the so-called immediate exchange model studied in \cite{heinsalu_kinetic_2014} assumes that pairs of agents are randomly and uniformly picked at each random time, and each of the agents transfer a random fraction of its money to the other agents, where these fractions are independent and uniformly distributed in $[0,1]$. The so-called uniform reshuffling model investigated in \cite{cao_entropy_2021,dragulescu_statistical_2000} requires that the total amount of money of two randomly and uniformly picked agents possess before interaction is uniformly redistributed among the two agents after interaction. The so-called unbiased exchange model and the rich-biased exchange proposed in \cite{cao_derivation_2021} are closely related variants of the poor-biased exchange model investigated in this work, where the variations of these models differ in the rate of an agent (say agent $i$) being picked to give out a dollar. Indeed, for the unbiased exchange dynamics and the rich-biased exchange dynamics, one modify the corresponding update rules \eqref{poor_biased_exchange} to (recall that if agent $i$ has no dollars to give, then nothing will happen)
\begin{equation}
\label{unbiased_biased_exchange}
\textbf{unbiased exchange:} \qquad (S_i,S_j)~ \begin{tikzpicture} \draw [->,decorate,decoration={snake,amplitude=.4mm,segment length=2mm,post length=1mm}]
(0,0) -- (.6,0); \node[above,red] at (0.3,0) {\small{$λ\slash N$}};\end{tikzpicture}~  (S_i-1,S_j+1)
\end{equation}
and \begin{equation}
\label{rich_biased_exchange}
\textbf{rich-biased exchange:} \qquad (S_i,S_j)~ \begin{tikzpicture} \draw [->,decorate,decoration={snake,amplitude=.4mm,segment length=2mm,post length=1mm}]
(0,0) -- (.6,0); \node[above,red] at (0.3,0) {\small{$λ\slash (N\,S_i)$}};\end{tikzpicture}~  (S_i-1,S_j+1),
\end{equation}
respectively. For other models arising from econophysics, we refer to \cite{chakraborti_statistical_2000,chatterjee_pareto_2004, lanchier_rigorous_2019,cao_explicit_2021,during_kinetic_2008,merle_cutoff_2019} and the references therein.

To the best of our knowledge, uniform in time propagation of chaos for other models coming from econophysics has only been studied in \cite{cortez_uniform_2016,cortez_quantitative_2016,cortez_particle_2018} which include the uniform reshuffling model and the immediate exchange model as special cases. The approach taken in \cite{cortez_uniform_2016,cortez_quantitative_2016,cortez_particle_2018} is an ``optimal-coupling'' type argument, which can be dated back to 1970s \cite{murata_propagation_1977,tanaka_probabilistic_1978} and which relies on a stochastic differential equation (SDE) representation of the agent-based stochastic dynamics in terms of Poisson random measures. A very recent work \cite{cao_interacting_2022} has established a uniform propagation of chaos result for the unbiased exchange model, based on a careful study of the entropy-entropy dissipation relation, at the level of the $N$-agent system as well as its associated mean-field system of nonlinear ODEs. Our approach to the uniform propagation of chaos for the poor-biased exchange model at hand will be built upon probabilistic coupling methods, and the non-uniform propagation of chaos shown in \cite{cao_derivation_2021}. 


\section{Uniform propagation of chaos}\label{sec:sec2}
\setcounter{equation}{0}

Throughout this section, we will employ the notation $\mathcal{L}({\bf X})$ to represent the law of a generic random variable or vector ${\bf X}$. We set $[N] \coloneqq \{1,2\ldots,N\}$ for notational simplicity, and we will also quantify convergence of probability measures via the Wasserstein distance (of order $1$): for probability measures $\mu, \nu$ on $\mathbb{N}^d$ with finite first moment, it is defined by
\begin{equation}\label{def:Wasserstein}
\mathcal{W}_1(\mu,\nu) \coloneqq \inf\limits_{{\bf X},{\bf Y}} \mathbb{E}\left[\frac{1}{d}\sum\limits_{j=1}^d |X_j - Y_j| \right],
\end{equation}
in which the infimum is taken over all possible couplings of $\mu$ and $\nu$, or equivalently over all pair of random vectors ${\bf X} = (X_1,\ldots,X_d)$ and ${\bf Y} = (Y_1,\ldots,Y_d)$ such that ${\bf X}$ and ${\bf Y}$ are distributed according to $\mu$ and $\nu$ respectively. Before we dive into the details of the proof of Theorem \ref{thm:uniform_poc_k_particle_marginal} on the uniform propagation of chaos, we summarize the spirit behind the proof in figure \ref{fig:sketch_proof_unif_chaos} below.

\begin{figure}[!htb]
\centering
\includegraphics[scale=0.7]{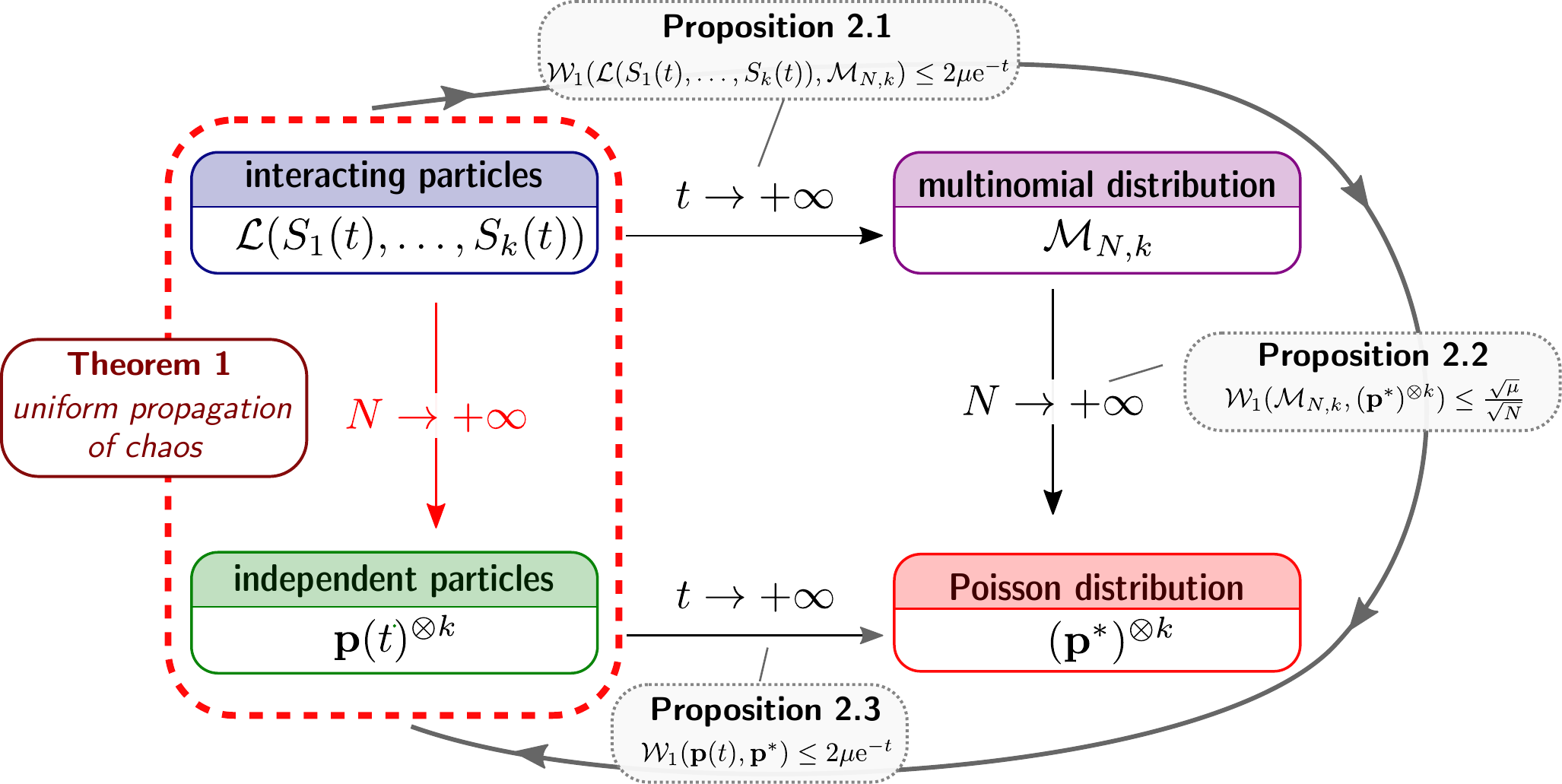}
\caption{Schematic illustration of the strategy behind the proof of the uniform in time propagation of chaos for the poor-biased dollar exchange model.}
\label{fig:sketch_proof_unif_chaos}
\end{figure}

To control $\mathcal{W}_1\left(\mathcal{L}\left(S_1(t),S_2(t),\ldots,S_k(t)\right), {\bf p}(t)^{\otimes k}\right)$ (where $1\leq k \leq N$ is fixed and does not grow as $N \to \infty$), we apply the triangle inequality twice as follows:
\begin{equation}\label{eq:roadmap}
\begin{aligned}
&\mathcal{W}_1\left(\mathcal{L}\left(S_1(t),S_2(t),\ldots,S_k(t)\right), {\bf p}(t)^{\otimes k}\right) \\
&\quad \leq  \mathcal{W}_1\left(\mathcal{L}\left(S_1(t),S_2(t),\ldots,S_k(t)\right), \mathcal{M}_{N,k}\right) + \mathcal{W}_1\left(\mathcal{M}_{N,k},({\bf p}^*)^{\otimes k}\right) \\
&\qquad + \mathcal{W}_1\left({\bf p}(t)^{\otimes k},({\bf p}^*)^{\otimes k}\right),
\end{aligned}
\end{equation}
where $\mathcal{M}_{N,k}$ is the $k$-particle marginal of $\mathcal{M}_N$ and ${\bf p}^*$ is the Poisson distribution with mean $\mu$. Section \ref{subsec:2.1} is devoted to the bound on $\mathcal{W}_1\left(\mathcal{L}\left(S_1(t),S_2(t),\ldots,S_k(t)\right), \mathcal{M}_{N,k}\right)$, Section \ref{subsec:2.2} and Section \ref{subsec:2.3} treat the bound on $\mathcal{W}_1\left(\mathcal{M}_{N,k},({\bf p}^*)^{\otimes k}\right)$ and $\mathcal{W}_1\left({\bf p}(t)^{\otimes k},({\bf p}^*)^{\otimes k}\right)$, respectively. Finally, the statement and the proof of the uniform propagation of chaos are presented in Section \ref{subsec:2.4}.

\subsection{Equilibration of the $N$-agents dynamics} \label{subsec:2.1}

One nice feature of the poor-biased exchange model is that, from the point of view of each individual dollar, the dynamics is very simple: each dollar jumps randomly and independently from pocket to pocket.
More specifically, we denote $M \coloneqq N\mu$ and introduce
\begin{equation}\label{eq:dollar-wise}
{\bf a}(t) \coloneqq \left(a_1(t),\ldots,a_M(t)\right) \in \{1,\ldots,N\}^M,
\end{equation}
in which $a_k(t)$ represents the (label of the) agent assigned to dollar $k$ at time $t$.
The dollar-wise poor-biased dynamics is specified as follows:
at each random time generated by an exponential clock with rate $\frac{N}{N-1}$, we pick a dollar $k \in \{1,\ldots,M\}$ and an agent $i \in \{1,\ldots,N\}$ independently and uniformly at random, then we update the label (or value) of $a_k$ to $i$. Notice that when $i=a_k$, nothing happens; this is why we set the rate to $\frac{N}{N-1}$: so that the \emph{effective} jump rate of each dollar is $1$. Consequently, when going back to the agent-wise dynamics (see \eqref{eq:agent-wise} below), one recovers exactly the same rates specified in \eqref{poor_biased_exchange}.

A natural way to couple two such processes ${\bf a}(t) \in \{1,\ldots,N\}^M$ and ${\bf b}(t) \in \{1,\ldots,N\}^M$ (with possibly different initial conditions) is to employ the same jump times, and pick the same dollar $k \in \{1,\ldots,M\}$ and agent $i \in \{1,\ldots,N\}$ in each update. To compare these processes, we will use the distance on $\{1,\ldots,N\}^M$ given by
\begin{equation}\label{eq:rho}
\rho({\bf a},{\bf b}) \coloneqq \sum\limits_{k=1}^M \mathbbm{1}\{a_k \neq b_k\}
\end{equation}
for all ${\bf a}, {\bf b} \in \{1,\ldots,N\}^M$. Now, if we let $\tau_k$ be the first time dollar $k$ is picked, then $\{\tau_k\}_{k=1}^M$ will be i.i.d with $\tau_k \sim \textrm{Exponential}(1)$, whence
\begin{align}
\notag
\mathbb{E}\,\rho\left({\bf a}(t),{\bf b}(t)\right)
&= \sum\limits_{k=1}^M \mathbb{P}\left(a_k(t) \neq b_k(t)\right) \\
\notag
&\leq \sum\limits_{k=1}^M \mathbb{P}\left(\tau_k > t\right) \\
\notag
&= M\,\expo^{-t} \\
\label{eq:estimate_rho}
&= N\,\mu\,\expo^{-t}.
\end{align}
This upper bound grows linearly with $N$, which in principle is unsatisfactory. However, we will now re-write this coupling in the setting of the agent-wise dynamics, and we will use an appropriate distance to compare these processes; then the estimate \eqref{eq:estimate_rho} will become useful.

Specifically, given ${\bf a}, {\bf b} \in \{1,\ldots,N\}^M$, we denote ${\bf Q}^{\bf a} = (Q^{\bf a}_1,\ldots,Q^{\bf a}_N) \in \mathscr{A}_{N,\mu}$ the vector of the number of dollars that each agent has according to ${\bf a}$, that is
\begin{equation}\label{eq:agent-wise}
Q^{\bf a}_i \coloneqq \sum\limits_{k=1}^M \mathbbm{1}\{a_k = i\},
\end{equation}
for all $1\leq i\leq N$, and similarly for ${\bf Q}^{\bf b} = (Q^{\bf b}_1,\ldots,Q^{\bf b}_N)$. We introduce the distance on $\mathscr{A}_{N,\mu}$ via
\begin{equation}\label{eq:d}
\dd\left({\bf S},{\bf R}\right) \coloneqq \frac{1}{N}\,\sum\limits_{i=1}^N \left|S_i - R_i\right|
\end{equation}
for all ${\bf S},{\bf R} \in \mathscr{A}_{N,\mu}$. Then:
\begin{align}
\notag
\dd\left({\bf Q}^{\bf a},{\bf Q}^{\bf b}\right) &= \frac{1}{N}\,\sum\limits_{i=1}^N\left|\sum\limits_{k=1}^M \mathbbm{1}\{a_k = i\} - \sum\limits_{k=1}^M \mathbbm{1}\{b_k = i\}\right| \\
\notag
&\leq \frac{1}{N}\,\sum\limits_{i=1}^N\sum\limits_{k=1}^M \left|\mathbbm{1}\{a_k = i\} -  \mathbbm{1}\{b_k = i\}\right| \\
\notag
&= \frac{1}{N}\,\sum\limits_{i=1}^N\sum\limits_{k=1}^M \left(\mathbbm{1}\{a_k = i, b_k \neq a_k\} + \mathbbm{1}\{b_k = i, a_k \neq b_k\}\right)\\
\notag
&= \frac{1}{N}\,\sum\limits_{k=1}^M \mathbbm{1}\{a_k \neq b_k\}\left(\sum\limits_{i=1}^N \mathbbm{1}\{a_k = i\} + \sum\limits_{i=1}^N \mathbbm{1}\{b_k = i\}\right) \\
\label{eq:dQaQb}
&= \frac{2}{N}\,\rho({\bf a},{\bf b}).
\end{align}
Now, define the processes ${\bf S}(t)$ and ${\bf R}(t)$ by
\[
{\bf S}(t) = {\bf Q}^{{\bf a}(t)}
\qquad \text{and} \qquad
{\bf R}(t) = {\bf Q}^{{\bf b}(t)},
\]
where $({\bf a}(t), {\bf b}(t))$ is the coupling defined above. It is straightforward to verify that ${\bf S}(t)$ and ${\bf R}(t)$ are indeed realizations of the poor-biased exchange model \eqref{poor_biased_exchange}. Consequently, using the estimates \eqref{eq:estimate_rho} and \eqref{eq:dQaQb}, we arrive at
\begin{align}
\notag
\mathcal{W}_1\left(\mathcal{L}\left({\bf S}(t)\right),\mathcal{L}\left({\bf R}(t)\right)\right) &\leq \mathbb{E}\,\dd\left({\bf S}(t),{\bf R}(t)\right) \\
\notag
&\leq \frac{2}{N}\,\mathbb{E}\,\rho\left({\bf a}(t),{\bf b}(t)\right) \\
&\leq 2\,\mu\,\expo^{-t}.
\label{eq:contraction_W1}
\end{align}

We summarize the previous discussions into the following proposition.

\begin{proposition}\label{prop:contraction_coupling}
Let ${\bf S}(t)$ and ${\bf R}(t)$ be two realizations of the poor-biased exchange model given by \eqref{poor_biased_exchange}, starting from any given pair of initial configurations ${\bf S}(0),{\bf R}(0) \in \mathscr{A}_{N,\mu}$. Then, we have
\[
\mathcal{W}_1\left(\mathcal{L}\left({\bf S}(t)\right),\mathcal{L}\left({\bf R}(t)\right)\right) \leq 2\,\mu\,\expo^{-t}.
\]
In particular, since the multinomial distribution $\mathcal{M}_N$
is the unique stationary distribution for the poor-biased exchange dynamics as $t \to \infty$ while $N$ is kept frozen (see for instance \cite{lanchier_rigorous_2017}), we also deduce that
\begin{equation}\label{eq:conver_to_multinomial}
\mathcal{W}_1\left(\mathcal{L}\left({\bf S}(t)\right),\mathcal{M}_N \right) \leq 2\,\mu\,\expo^{-t}.
\end{equation}
\end{proposition}

\begin{remark}
As any coupling of the full vector gives rise to a coupling of the marginals, we immediately deduce from Proposition \ref{prop:contraction_coupling} that
\begin{equation}\label{eq:conver_to_possion}
\mathcal{W}_1\left(\mathcal{L}\left(S_1(t)\right),\textrm{Binomial}\left(\mu N, \frac{1}{N}\right)\right) \leq 2\,\mu\,\expo^{-t}.
\end{equation}
Denoting by $\mathcal{M}_{N,k}$ the $k$-particle marginal distribution of $\mathcal{M}_N$ for each fixed $k \geq 1$, one also has
\begin{equation}\label{eq:conver_to_possion_k_particle}
\mathcal{W}_1\left(\mathcal{L}\left(S_1(t),S_2(t),\ldots,S_k(t)\right),\mathcal{M}_{N,k}\right) \leq 2\,\mu\,\expo^{-t}.
\end{equation}
\end{remark}

\subsection{Chaos of the multinomial to Poisson}\label{subsec:2.2}

In this subsection, we will establish a quantitative bound on the Wasserstein distance between the multinomial distribution $\mathcal{M}_N$ and the tensorized Poisson distribution $\textrm{Poisson}(\mu)^{\otimes N}$, which might be of independent interest. We recall that $\mathcal{M}_N$, rigorously defined by \eqref{eq:multinomial}, is the distribution corresponding to the experiment of tossing $N \mu$ balls independently in $N$ equally likely urns.

Our goal is to show the following bound:

\begin{proposition}
\label{prop:chaos_of_the_multinomial}
For each $\mu \in \mathbb{N}_+$ and each $N \geq 2$, we have
\begin{equation}
\label{eq:W1_multi_possion}
\mathcal{W}_1\left(\mathcal{M}_N,\textrm{Poisson}(\mu)^{\otimes N}\right)
\leq \frac{\sqrt{2\mu/\pi}}{\sqrt{N}}
\leq \frac{\sqrt{\mu}}{\sqrt{N}}.
\end{equation}
\end{proposition}

\begin{proof}
We proceed by a coupling argument: the idea is to define random vectors
\[
{\bf X}
= \begin{pmatrix}
X_1 \\ X_2 \\ \vdots \\ X_N
\end{pmatrix}
\sim \mathcal{M}_N,
\qquad
{\bf Y}
= \begin{pmatrix}
Y_1 \\ Y_2 \\ \vdots \\ Y_N
\end{pmatrix}
\sim \textrm{Poisson}(\mu)^{\otimes N}
\]
in a convenient way. Specifically, for each bin $i = 1,2,\ldots, N$, we consider an infinite collection of $\textrm{Exponential}(1)$ distributed independent random variables, stacked together. That is, we have $N$ independent copies of a one-dimensional Poisson process (see Figure \ref{fig:illustration_coupling_abstract} for an illustration). Every interval represents a ball that can potentially be tossed in the bin underneath.

\begin{figure}[!htb]
  \centering
  \includegraphics[scale = 1.2]{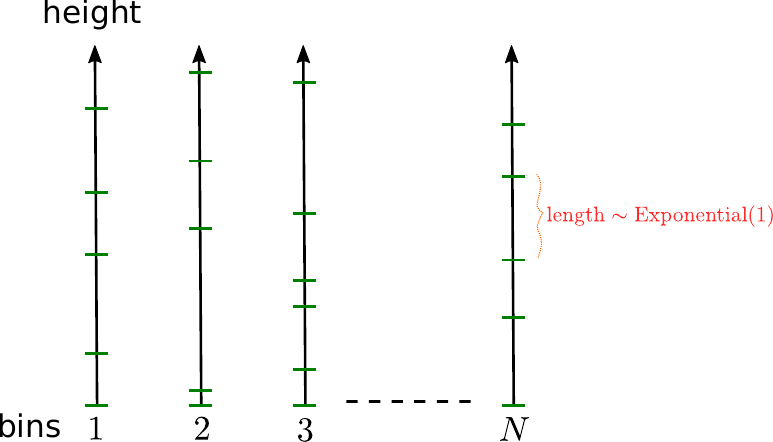}
  \caption{Construction of $N$ independent copies of a one-dimensional Poisson process via an infinite collection of i.i.d. exponentially distributed random variables.}
  \label{fig:illustration_coupling_abstract}
\end{figure}

Now, starting from height $0$, start rising a bar covering all bins, until exactly $\mu N$ full intervals lie below the bar. Call
\[
X_i
\coloneqq \textrm{number of full intervals below this bar at bin $i$}
\]
for each $1\leq i \in N$. It is readily seen that ${\bf X} \sim \mathcal{M}_N$, because ``rising a bar'' is the same as adding new balls at random among the $N$ bins, thanks to the loss of memory property. Also, let
\[
Y_i \coloneqq
\textrm{number of full intervals below height $\mu$ at bin $i$}
\]
for each $1\leq i \in N$. Clearly, $Y_1, Y_2,\ldots, Y_N$ are independent $\textrm{Poisson}(\mu)$ distributed random variables. See Figure \ref{fig:illustration_coupling_example} for a concrete illustration of the constructed coupling.

\begin{figure}[!htb]
  \centering
  \includegraphics[scale = 0.6]{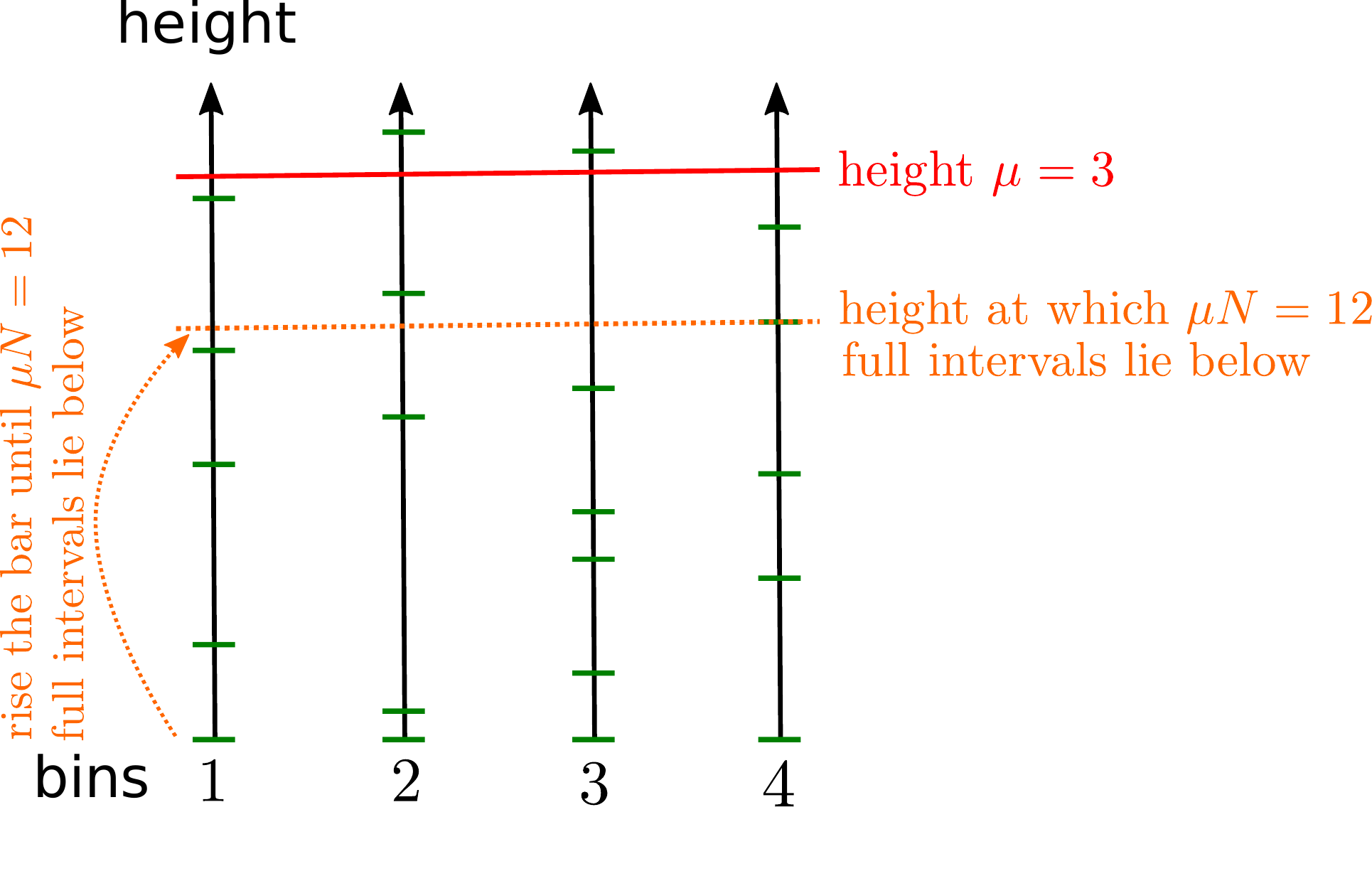}
  \caption{Coupling of the two random vectors ${\bf X} \sim \mathcal{M}_N$ and ${\bf Y} \sim \textrm{Poisson}(\mu)^{\otimes N}$ with $N = 4$ and $\mu = 3$. In this example, ${\bf X} = (X_1 = 4, X_2 = 3, X_3 = 4, X_4 = 4)$ and ${\bf Y} = (Y_1 = 3, Y_2 = 2, Y_3 = 4, Y_4 = 3)$.}
  \label{fig:illustration_coupling_example}
\end{figure}

Observe that to go from ${\bf X}$ to ${\bf Y}$ (or vice versa), we either add balls (i.e., full intervals) to some bins or remove balls from some bins, but we never add balls to some bins and remove balls from other bins simultaneously. This means that $X_i - Y_i$ has the same sign for all $i = 1,2,\ldots, N$. Also, recall that for a random variable $Z \sim \textrm{Poisson}(m)$ with $m$ integer, we have the following exact formula for the expected mean deviation:
\[
\mathbb{E}[|Z-m|]
= 2 \expo^{-m} \frac{m^{m+1}}{m!},
\]
see for instance \cite{crow_mean_1958,ramasubban_mean_1958}. Consequently, defining $Z \coloneqq \sum_{i=1}^N Y_i \sim \textrm{Poisson}(\mu N)$, we deduce that
\begin{equation}
\label{eq:can_be_improved}
\begin{aligned}
\mathbb{E}\left[\frac{1}{N}\,\sum\limits_{i=1}^N |X_i - Y_i|\right]
&= \frac{1}{N}\,\mathbb{E}\left[\left|\sum\limits_{i=1}^N X_i - \sum\limits_{i=1}^N Y_i\right|\right] \\
&= \frac{1}{N}\,\mathbb{E}[|\mu N - Z|] \\
&= \frac{1}{N} 2 \expo^{-\mu N} \frac{(\mu N)^{\mu N + 1}}{(\mu N)!} \\
&\leq \frac{\sqrt{2\mu/\pi}}{\sqrt{N}},
\end{aligned}
\end{equation}
where we have used the bound $n! \geq \sqrt{2 \pi n}\,(n / \expo)^n$, coming from Stirling's approximation. The desired bound now follows.
\end{proof}

\begin{remark}
Since a coupling of the marginals can be induced from a coupling of the full vector, Proposition \ref{prop:chaos_of_the_multinomial} implies that
\begin{equation}\label{eq:W1_binom_possion}
\mathcal{W}_1\left(\textrm{Binomial}\left(\mu N, \frac{1}{N}\right),\textrm{Poisson}(\mu)\right)
\leq \frac{\sqrt{2\mu/\pi}}{\sqrt{N}}
\leq \frac{\sqrt{\mu}}{\sqrt{N}},
\end{equation}
as well as
\begin{equation}\label{eq:W1_k_nomial_possion}
\mathcal{W}_1\left(\mathcal{M}_{N,k},\textrm{Poisson}(\mu)^{\otimes k}\right)
\leq \frac{\sqrt{2\mu/\pi}}{\sqrt{N}}
\leq \frac{\sqrt{\mu}}{\sqrt{N}},
\end{equation}
for each fixed $1\leq k \leq N$.
\end{remark}


\subsection{Equilibration of the mean-field ODE towards Poisson}\label{subsec:2.3}

We are now ready to establish the last piece of result before we prove the desired uniform in time propagation of chaos. We first recall that a non-uniform propagation of chaos result for the poor-biased exchange model has already been obtained in \cite[Theorem 2]{cao_derivation_2021}: assuming $\mathcal{L}\left(S_1(0)\right) = {\bf p}(0)$, then for each $t \geq 0$ and each $N \geq 2$, we have
\begin{equation}
\label{eq:non_uniform_chaos}
\mathcal{W}_1\left(\mathcal{L}\left(S_1(t)\right), {\bf p}(t)\right) \leq \frac{4\,\mu\,\expo^t}{N},
\end{equation}
where ${\bf p}(t)$ is the unique solution to the mean-field system of linear ODEs \eqref{eq:law_limit}. Combining this result with Propositions \ref{prop:contraction_coupling} and \ref{prop:chaos_of_the_multinomial} leads us to the following estimate:

\begin{proposition}\label{prop:equilibration_of_ODE}
Suppose that ${\bf p}(t)=\{p_n(t)\}_{n\geq 0}$ is the unique solution of \eqref{eq:law_limit} and ${\bf p}^*$ is the equilibrium Poisson distribution \eqref{eqn:equil_limit} to \eqref{eq:law_limit}. Then for each $t \geq 0$,
\begin{equation}\label{eq:large_time_convergence}
\mathcal{W}_1\left({\bf p}(t), {\bf p}^*\right) \leq 2\,\mu\,\expo^{-t}.
\end{equation}
\end{proposition}

\begin{proof} An elementary application of the triangle inequality for the Wasserstein distance $W_1$, together with the estimates \eqref{eq:conver_to_possion}, \eqref{eq:W1_binom_possion} and \eqref{eq:non_uniform_chaos}, yields
\[
\begin{aligned}
\mathcal{W}_1\left({\bf p}(t), {\bf p}^*\right) &\leq \mathcal{W}_1\left({\bf p}(t), \mathcal{L}\left(S_1(t)\right)\right) + \mathcal{W}_1\left(\mathcal{L}\left(S_1(t)\right),\textrm{Binomial}\left(\mu N, \frac{1}{N}\right)\right) \\
&\qquad + \mathcal{W}_1\left(\textrm{Binomial}\left(\mu N, \frac{1}{N}\right),{\bf p}^*\right) \\
&\leq \frac{4\,\mu\,\expo^t}{N} + 2\,\mu\,\expo^{-t} + \frac{\sqrt{\mu}}{\sqrt{N}}.
\end{aligned}
\]
As this estimate is valid for any $N$, sending $N \to \infty$ gives rise to the desired bound \eqref{eq:large_time_convergence}.
\end{proof}

\begin{remark}
It is very natural to expect that a result of the type \eqref{eq:large_time_convergence}, which concerns only the large-time behavior of the mean-field ODE system \eqref{eq:law_limit}, can be obtained in a purely analytic way without resorting to any probabilistic argument of the underlying stochastic agent-based dynamics. Indeed, it has already been shown in \cite{cao_derivation_2021} that the so-called $\chi^2$ ``distance'' from ${\bf p}(t)$ to ${\bf p}^*$, defined by
\begin{equation*}
\chi^2\left({\bf p}(t),{\bf p}^*\right) \coloneqq \sum\limits_{n=0}^\infty \frac{|p_n(t) - p^*_n|^2}{p^*_n},
\end{equation*}
satisfies
\begin{equation}
\label{eq:exponential_decay_chi2}
\chi^2\left({\bf p}(t),{\bf p}^*\right) \leq \chi^2\left({\bf p}(0),{\bf p}^*\right)\,\expo^{-t}.
\end{equation}
Taking into account the possibility to bound the Wasserstein distance $\mathcal{W}_1\left({\bf p}(t), {\bf p}^*\right)$ by the $\chi^2$ ``distance'' $\chi^2\left({\bf p}(t),{\bf p}^*\right)$, such as (see for instance \cite{pinelis_inequality_2022} for its proof)
\begin{equation}\label{eq:bound_W1_by_chi2}
\mathcal{W}_1\left({\bf p}(t), {\bf p}^*\right) \leq \sqrt{\mu^2 + \mu}\,\sqrt{\chi^2\left({\bf p}(t),{\bf p}^*\right)},
\end{equation}
we obtain
\begin{equation}\label{eq:large_time_convergence_version2}
\mathcal{W}_1\left({\bf p}(t), {\bf p}^*\right) \leq \sqrt{\mu^2 + \mu}\,\sqrt{\chi^2\left({\bf p}(0),{\bf p}^*\right)}\,\expo^{-\frac{t}{2}}.
\end{equation}
Notice that \eqref{eq:large_time_convergence} is sharper than \eqref{eq:large_time_convergence_version2}. However, it is conjectured in \cite{cao_derivation_2021} that $\chi^2\left({\bf p}(t),{\bf p}^*\right)$ decays like $\expo^{-2\,t}$ (based on some heuristic reasoning and numerical experiments), which would then lead to an estimate similar to \eqref{eq:large_time_convergence}.
\end{remark}

\subsection{Proof of the uniform propagation of chaos}\label{subsec:2.4}

We now assemble all the previous estimates together to prove a uniform-in-time propagation of chaos result for the poor-biased exchange dynamics. The key idea behind the proof is to carefully choose the time $t$ as a function of the number of agents $N$, inspired from a recent work \cite{cao_interacting_2022} on a closely-related model.

We will need the following general version of \eqref{eq:non_uniform_chaos}, see \cite[Theorem 2]{cao_derivation_2021}: assuming that $S_1(t),\ldots,S_N(t)$ are i.i.d.\ with law ${\bf p}(0)$, then for any fixed number of marginals $k$, we have
\begin{equation}\label{eq:non_uniform_chaos_k_particles}
\mathcal{W}_1\left(\mathcal{L}\left(S_1(t),\ldots,S_k(t)\right), {\bf p}(t)^{\otimes k}\right) \leq \frac{4\,k\,\mu\,\expo^t}{N}.
\end{equation}

\begin{theorem}
\label{thm:uniform_poc_k_particle_marginal}
Assume that $S_1(t),\ldots,S_N(t)$ are i.i.d.\ with law ${\bf p}(0)$. Then, for all fixed $k \geq 1$ and for all $N \geq 2$ and $t \geq 0$, we have
\begin{equation}
\label{eq:uniform_poc_k_particle_marginal}
\mathcal{W}_1\left(\mathcal{L}\left(S_1(t),\ldots,S_k(t)\right), {\bf p}(t)^{\otimes k}\right) \leq \frac{4\,k\,\mu + \sqrt{\mu}}{\sqrt{N}}.
\end{equation}
\end{theorem}

\begin{proof}
Notice that the appropriate scaling in the definition of Wasserstein distance \eqref{def:Wasserstein} ensures that a tensorized version of \eqref{eq:large_time_convergence} remains valid as well, i.e.,
\begin{equation}\label{eq:large_time_convergence_k_particle}
\mathcal{W}_1\left({\bf p}(t)^{\otimes k}, ({\bf p}^*)^{\otimes k}\right) \leq 2\,\mu\,\expo^{-t}.
\end{equation}
Thus, from \eqref{eq:conver_to_possion_k_particle}, \eqref{eq:W1_k_nomial_possion} and \eqref{eq:large_time_convergence_k_particle}, we have
\begin{align*}
&\mathcal{W}_1\left(\mathcal{L}\left(S_1(t),\ldots,S_k(t)\right), {\bf p}(t)^{\otimes k}\right) \\
&\leq \mathcal{W}_1\left(\mathcal{L}\left(S_1(t),\ldots,S_k(t)\right), \mathcal{M}_{N,k} \right)
+ \mathcal{W}_1\left(\mathcal{M}_{N,k}, ({\bf p}^*)^{\otimes k} \right)
+ \mathcal{W}_1\left( ({\bf p}^*)^{\otimes k}, {\bf p}(t)^{\otimes k} \right) \\
&\leq 4\,\mu\,\expo^{-t} + \frac{\sqrt{\mu}}{\sqrt{N}}.
\end{align*}
Now, set $T = \frac{\log N}{2}$. Using the last estimate when $t\geq T$, and \eqref{eq:non_uniform_chaos_k_particles} when $t\leq T$, gives
\begin{equation}
\label{eq:two_cases}
\mathcal{W}_1\left(\mathcal{L}\left(S_1(t),\ldots,S_k(t)\right), {\bf p}(t)^{\otimes k}\right) \leq
    \begin{cases}
        \frac{4\,k\,\mu}{\sqrt{N}} & \text{if } t \leq T,\\
        \frac{4\,\mu + \sqrt{\mu}}{\sqrt{N}} & \text{if } t \geq T,
    \end{cases}
\end{equation}
from which \eqref{eq:uniform_poc_k_particle_marginal} follows.
\end{proof}

\section{Conclusion}
\setcounter{equation}{0}

In this manuscript, an agent-based dollar exchange model (called the poor-biased exchange model in \cite{cao_derivation_2021}) for wealth (re-)distribution is studied. We rigorously proved a uniform in time propagation of chaos result for this model which, to the best of our knowledge, is not available in the literature prior to the present work. Our proof is based on several probabilistic coupling approaches and the non-uniform propagation of chaos established in \cite{cao_derivation_2021}, as well as some ideas from the recent work \cite{cao_interacting_2022} for a closely-related model. We emphasize that the poor-biased exchange model investigated in this paper has at least two ``siblings'', known as the unbiased exchange model and the rich-biased exchange model \cite{cao_derivation_2021}, where the rate that a typical agent will be picked to give are modified according to \eqref{unbiased_biased_exchange} \eqref{rich_biased_exchange}, respectively. One possible follow-up work would be to have a rigorous proof of the sharp estimate (recall \eqref{eq:exponential_decay_refined} or equivalently \eqref{eq:exponential_decay_chi2}) for the solution of the mean-field ODE system \eqref{eq:law_limit} in the $\chi^2$ ``distance'' conjectured in \cite{cao_derivation_2021}.

As of now, one exciting direction of research for econophysics models involves the inclusion of a central bank or even several banks and hence the possibility of agents being in debt (see for instance the recent work \cite{cao_uncovering_2022,lanchier_rigorous_2019}). We plan to extend the framework and analysis of the present work to this more realistic setting.

\end{document}